\documentclass{amsart}
\usepackage{color}
\usepackage{graphicx,epsf}
\usepackage{amsmath,amscd}
\newtheorem{theorem}{Theorem}[section]
\newtheorem{lemma}[theorem]{Lemma}

\theoremstyle{definition}
\newtheorem{definition}[theorem]{Definition}

\newtheorem{notation}[theorem]{Notation}

\theoremstyle{remark}
\newtheorem{remark}[theorem]{Remark}
\numberwithin{equation}{section}

\newcommand{\Real}{{\mathbb R}}

\newcommand{\eps}{\varepsilon}
\newcommand{\f}{\mathbf{f}}
\newcommand{\x}{\mathbf{x}}

\newcommand {\HH} {{\rm H}}

\newcommand {\card} {{\rm card}}
\newcommand {\hide}[1]{}

\begin{document}
\title[A Helly-type theorem for semi-monotone sets and monotone maps]
{A Helly-type theorem  for semi-monotone sets and monotone maps}
\author{Saugata Basu}
\address{Department of Mathematics,
Purdue University, West Lafayette, IN 47907, USA}
\email{sbasu@math.purdue.edu}
\author{Andrei Gabrielov}
\address{Department of Mathematics,
Purdue University, West Lafayette, IN 47907, USA}
\email{agabriel@math.purdue.edu}
\author{Nicolai Vorobjov}
\address{
Department of Computer Science, University of Bath, Bath
BA2 7AY, England, UK}
\email{nnv@cs.bath.ac.uk}
\thanks{The first author was supported in part by NSF grant CCF-0915954.
The second author was supported in part by NSF grants DMS-0801050 and DMS-1067886}

\begin{abstract}
\hide{
Semi-monotone sets and monotone maps are special classes
of definable sets and maps respectively in o-minimal structures over $\Real$.
They were defined to serve as building blocks for
obtaining cylindrical cell decomposition of given definable sets
without changing the coordinate system. By definition, semi-monotone
sets and graphs of monotone maps, share certain properties of
those convex subsets of $\Real^n$ which are sub-manifolds of linear manifolds.
However, semi-monotone
sets and graphs of monotone maps are in general very  far from being convex,
and intersections amongst such sets, unlike convex sets, can be topologically
complicated. Despite this lack of good intersectional properties, we prove
a Helly-type theorem for these special classes of definable sets.}

We consider sets and maps defined over an o-minimal structure over the reals,
such as real semi-algebraic or 
globally
subanalytic sets.
A {\em monotone map} is a multi-dimensional generalization of a usual univariate monotone 
continuous function on an open interval,
while the closure of the graph of a monotone map is a generalization of a compact convex set.
In a particular case of an identically constant function, such a 
graph is called a {\em semi-monotone set}.
Graphs of monotone maps are, generally, non-convex,
and their intersections, unlike intersections of convex sets, can be topologically complicated.
In particular, such 
an intersection  is not necessarily the graph of a monotone map.
Nevertheless, we prove a Helly-type theorem, which says that for a finite family of subsets of
$\Real^n$, if all intersections of
subfamilies, with cardinalities at most $n+1$, are non-empty and graphs of monotone maps, then
the intersection of the whole family is non-empty and the graph of a monotone map.
\end{abstract}
\maketitle

\section{Introduction}
In \cite{BGV,BGV2012} the authors introduced a certain
class of definable subsets of $\Real^n$ (called \emph{semi-monotone sets})
and definable maps $f: \Real^n \rightarrow \Real^k$ (called
\emph{monotone maps}) in an o-minimal structure over $\Real$.
These objects are meant to serve as building blocks for obtaining 
a conjectured cylindrical cell decomposition 
of definable sets into topologically regular cells,
without changing the coordinate system in the ambient space $\Real^n$
(see \cite{BGV,BGV2012} for a more detailed motivation behind these 
definitions).

The semi-monotone sets, and more generally the graphs of monotone maps,
have certain properties which resemble those of classical convex subsets
of $\Real^n$.
Indeed, the intersection
of any definable open convex subset of $\Real^n$ with an affine
flat (possibly $\Real^n$ itself) is the graph of a monotone map.
In this paper, we prove a version of the classical theorem of Helly on
intersections of convex subsets of $\Real^n$.

We first fix some notation that we are going to use
for the rest of the paper.

\begin{notation}
For every positive integer $p$, we will denote by $[p]$ the set $\{1,\ldots,p\}$.
We fix an integer $s > 0$, and we will henceforth denote by $I$ the set $[s]$.
For any family, $\mathcal{F} = (\mathbf{F}_i)_{i \in I}$, of subsets of $\Real^n$
and $J \subset I$, we will denote by $\mathcal{F}_J$ the set
$\displaystyle{ \bigcap_{j \in J} \mathbf{F}_j}$.
\end{notation}

\begin{theorem}[Helly's Theorem \cite{Helly,Radon}]\label{the:classical_helly}
Let $\mathcal{F} = (\mathbf{F}_i)_{i \in I}$ be a family of convex subsets of $\Real^n$, such
that for each subset $J \subset I$ such that $\card\; J \leq n+1$, the intersection $\mathcal{F}_J$
is non-empty.
Then, $\mathcal{F}_I$ is non-empty.
\end{theorem}

In this paper we prove an analogue of Helly's theorem for semi-monotone sets as well as for
graphs of monotone maps.
One important result in 
\cite[Theorem 13]{BGV2012} 
is that the graph
of a monotone map is a topologically regular cell.
However, unlike the case of a family of convex sets, 
the intersection of a finite
family of graphs of monotone maps need not be a graph of a monotone map, or even be connected.
Moreover, such an intersection can have an arbitrarily large number of connected components.
Because of this lack of a good intersectional property, one would not normally expect
a Helly-type theorem to hold in this case.
Nevertheless, we are able to prove the following theorem.

\begin{theorem}
\label{the:monotone_helly}
Let $\mathcal{F} = (\mathbf{F}_i)_{i \in I}$
be a family of definable subsets of $\Real^n$ such
that for each $i \in I$ the set $\mathbf{F}_i$  is the graph of a monotone map, and
for each $J \subset I$, with $\card\; J \leq n+1$, the intersection $\mathcal{F}_J$
is non-empty and the graph of a monotone map.
Then, $\mathcal{F}_I$ is non-empty and the graph of a monotone map as well.

Moreover, if $\dim \; \mathcal{F}_J \geq d$
for each $J \subset I$, with $\card\; J \leq n+1$, then $\dim \; \mathcal{F}_I \geq d$.
\end{theorem}

\begin{remark}
Katchalski \cite{Katchalski71}
(see also \cite{Grunbaum}) proved the following generalization of
Helly's theorem which took into account dimensions of the various
intersections.

\begin{theorem}[\cite{Katchalski71,Grunbaum}]
\label{the:grunbaum}
Define the function $g(j)$ as follows:
\begin{itemize}
\item[ ]
$g(0) = n+1$,
\item[ ]
$g(j) = \max(n+1, 2(n-j+1))$ for $1 \leq j \leq  n$.
\end{itemize}
Fix any $j$ such that $0 \leq j \leq n$.
Let $\mathcal{F} = (\mathbf{F}_i)_{i \in I}$
be a family of convex subsets of $\Real^n$, with $\card\; I \geq g(j)$,
such that for each $J \subset I$, with $\card\; J \leq g(j)$, the dimension
$\dim \; \mathcal{F}_J \geq j$.
Then, the dimension $\dim \; \mathcal{F}_I \geq j$.
\end{theorem}

Notice that in the special case of definable convex sets in $\Real^n$ that are open subsets
of flats, Theorem \ref{the:monotone_helly} gives a slight improvement over
Theorem \ref{the:grunbaum} in that $n+1 \leq g(j)$ for all $j,\> 0 \leq j \leq n$,
where $g(j)$ is the function defined in Theorem \ref{the:grunbaum}.
The reason  behind this improvement is that convex sets that are graphs of monotone maps (i.e.,
definable open convex subsets of affine flats) are rather special and
easier to deal with,  since we do not need to control the intersections of their boundaries.

Also note that, while it follows immediately from Theorem \ref{the:grunbaum}
(using the same notation) that
\[
\dim \; \mathcal{F}_I = \min(\dim \; \mathcal{F}_J | J \subset I, \card \; J \leq 2n),
\]
Katchalski \cite{Katchalski78}  proved the stronger statement that 
\[
\dim \; \mathcal{F}_I = \min(\dim \; \mathcal{F}_J| J \subset I, \card \; J \leq n+1).
\]
In the case of graphs of monotone maps, the analogue of the latter  
statement is an immediate consequence of Theorem \ref{the:monotone_helly}.   
\end{remark}

\section{Proof of Theorem \ref{the:monotone_helly}}
We begin with a few preliminary definitions.

\begin{definition}
\label{def:semi-monotone}
Let $L_{j, \sigma, c}:= \{ \x=(x_1, \ldots ,x_n) \in \Real^n|\> x_j \sigma c \}$
for $j=1, \ldots ,n$, $\sigma \in \{ <,=,> \}$, and $c \in \Real$.
Each intersection of the kind
$$C:=L_{j_1, \sigma_1, c_1} \cap \cdots \cap L_{j_m, \sigma_m, c_m} \subset \Real^n,$$
where $m=0, \ldots ,n$, $1 \le j_1 < \cdots < j_m \le n$, $\sigma_1, \ldots ,\sigma_m \in \{<,=,> \}$,
and $c_1, \ldots ,c_m \in \Real$, is called a {\em coordinate cone} in $\Real^n$.

Each intersection of the kind
$$S:=L_{j_1, =, c_1} \cap \cdots \cap L_{j_m, =, c_m} \subset \Real^n,$$
where $m=0, \ldots ,n$, $1 \le j_1 < \cdots < j_m \le n$,
and $c_1, \ldots ,c_m \in \Real$, is called an {\em affine coordinate subspace} in $\Real^n$.

In particular, the space $\Real^n$ itself is both a coordinate cone and an affine coordinate
subspace in $\Real^n$.
\end{definition}

\begin{definition}[\cite{BGV}]\label{def:set}
An open (possibly, empty) bounded set $X \subset \Real^n$ is called {\em semi-monotone} if
for each coordinate cone $C$  the intersection $X \cap C$ is connected.
\end{definition}

\begin{remark}
In fact, in Definition \ref{def:set} above, it suffices to consider
intersections with only affine
coordinate subspaces (see 
Definition~\ref{def:def_monotone_map} 
below). 
\end{remark}

We refer the reader to \cite[Figure 1]{BGV} for some examples of
semi-monotone subsets of $\Real^2$, as well as some counter-examples.
In particular, it is clear from the examples that the intersection of two semi-monotone sets
in plane is not necessarily connected and hence not semi-monotone.

Notice that any 
bounded
convex open subset of $\Real^n$ is semi-monotone.

\hide{
The definition of \emph{monotone maps} is given in \cite{BGV2012} and is a bit more technical.
We will not repeat it here but recall a few important properties of monotone maps that we will need.
In particular, Theorem~\ref{th:def_monotone_map} below, which appears in \cite{BGV2012},
gives a complete characterization of monotone maps.
For the purposes
of the present paper this characterization can be taken as the definition of monotone maps.
}

We now define \emph{monotone maps}.
The definition below is not the
one given in \cite{BGV2012}, but equivalent to it as shown in
\cite[Theorem~9]{BGV2012}.

We first need a preliminary definition.
\begin{definition}\label{def:quasi-affine}
Let a bounded continuous map $\f=(f_1, \ldots ,f_k)$ defined on an open bounded non-empty set
$X \subset \Real^n$ have the graph ${\bf F} \subset \Real^{n+k}$.
We say that $\f$ is {\em quasi-affine} if for any coordinate subspace
$T$ of $\Real^{n+k}$, the projection $\rho_T:\> {\bf F} \to T$ is injective if and only if the image
$\rho_T({\bf F})$ is $n$-dimensional.
\end{definition}

\begin{definition}
\label{def:def_monotone_map}
Let a bounded continuous quasi-affine map $\f=(f_1, \ldots ,f_k)$ defined on an open bounded
non-empty set $X \subset \Real^n$ have the graph ${\bf F} \subset \Real^{n+k}$.
We say that the map $\f$ is monotone if 
for each affine coordinate subspace $S$ in $\Real^{n+k}$ the intersection 
${\bf F} \cap S$ is connected.
\end{definition}

\begin{remark}
Notice that it follows immediately from Definition \ref{def:set} that the domain of
definition $X \subset \Real^n$ in Definition \ref{def:def_monotone_map} above is a
semi-monotone set. In particular, also notice that the graphs of monotone maps that appear
in the statement of  Theorem \ref{the:monotone_helly}
refer to graphs of maps defined on open bounded domains, which are semi-monotone subsets of
affine coordinate subspaces of $\Real^n$.
\end{remark}

The following two statements were proved in \cite{BGV2012}.

\begin{theorem}\cite[Corollary~7]{BGV2012}
\label{th:def_monotone_map}
Let $\f:\> X \to \Real^k$ be a monotone map having the graph ${\bf F} \subset \Real^{n+k}$.
Then for every coordinate $z$ in $\Real^{n+k}$ and every $c \in \Real$,
each of the intersections ${\bf F} \cap \{ z\> \sigma\> c \}$, where $\sigma \in \{ <,>,= \}$,
is either empty or the graph of a monotone map.
\end{theorem}

\begin{theorem}\cite[Theorem~10]{BGV2012}
\label{th:proj}
Let $\f:\> X \to \Real^k$ be a monotone map defined on a semi-monotone set $X \subset \Real^n$ and
having the graph ${\bf F} \subset \Real^{n+k}$.
Then for any coordinate subspace $T$ in $\Real^{n+k}$
the image $\rho_T ({\bf F})$ under the projection map $\rho_T:\> {\bf F} \to T$
is either a semi-monotone set or the graph of a monotone map.
\end{theorem}

\begin{remark}
\label{rem:quasi-affine}
In view of 
Definition~\ref{def:def_monotone_map}, 
it is natural to identify any semi-monotone set
$X \subset \Real^n$ with the graph of 
the constant function $f:X\rightarrow \Real^0 = {\mathbf{0}}$.
Also, note that in this case the function $f$ is trivially quasi-affine (cf. Definition \ref{def:quasi-affine}).
\end{remark}

We need two preliminary lemmas before we prove Theorem
\ref{the:monotone_helly}.

\begin{lemma}
\label{lem:intermediate}
Suppose that
$\mathcal{F} = (\mathbf{F}_i)_{i \in I}$
is a family of definable subsets of $\Real^n$ such
that for each $i \in I$ the set $\mathbf{F}_i$ is the graph of a monotone map.
Then, there exists a family of definable sets,
$\mathcal{F}' = (\mathbf{F}_i')_{i \in I}$ such that:
\begin{enumerate}
\item
for each $i\in I$ the set $\mathbf{F}_i'$ is 
compact;
\item
for each $J \subset I$ we have
$$
\displaylines{
\HH_*(\mathcal{F}'_J,\mathbb{Z})
\cong
\HH_*(\mathcal{F}_J,\mathbb{Z}),
}
$$
where $\HH_*(X,\mathbb{Z})$ denotes the singular homology of $X$.
\end{enumerate}
\end{lemma}

\begin{proof}
Since, according to 
\cite[Theorem~13]{BGV2012} 
the graph of a monotone map is a regular cell,
we have for each $i \in I$ a definable
homeomorphism
$$\phi_i: (0,1)^{\dim(\mathbf{F}_i)} \rightarrow \mathbf{F}_i.$$

For each real $\eps > 0$ small enough, and for each $i \in I$ consider the image
$$\mathbf{F}_{i}^{(\eps)} = \phi_i([\eps,1-\eps]^{\dim(\mathbf{F}_i)}).$$
Consider the family $\mathcal{F}^{(\eps)} = \left( \mathbf{F}_{i}^{(\eps)} \right)_{i \in I}$.
Observe for each $J \subset I$ and each $\eps >0$ the intersection
$\mathcal{F}^{(\eps)}_{J}$ is compact, and the increasing family
$\left(\mathcal{F}^{(\eps)}_{J} \right)_{\eps> 0}$
is co-final in the directed system (under the inclusion maps) of the compact subsets of $\mathcal{F}_J$.
Since, the singular homology group of any space is isomorphic to the direct limit of the singular
homology groups of its compact subsets \cite[Sec. 4, Theorem 6]{Spanier}, we have
$$
\displaylines{
\varinjlim \HH_*(\mathcal{F}^{(\eps)}_{J}, \mathbb{Z})
\cong
\HH_*(\mathcal{F}_{J},\mathbb{Z}).
}
$$
Finally, by Hardt's triviality theorem \cite{Michel2} there exists
$\eps_0 > 0$, such  that
$$
\displaylines{
\varinjlim \HH_*(\mathcal{F}^{(\eps)}_{J}, \mathbb{Z})
\cong
\HH_*(\mathcal{F}^{(\eps_0)}_{J},\mathbb{Z}).
}
$$
For each $i \in I$, we let $\mathbf{F}_i' = \mathbf{F}_{i}^{(\eps_0)}$.
\end{proof}

\begin{lemma}
\label{lem:intermediate2}
Let $\mathcal{F} = (\mathbf{F}_i)_{i \in I}$ be a family of definable subsets of $\Real^n$ such
that for each $i \in I$ the set $\mathbf{F}_i$ is the graph of a monotone map, and
for each $J \subset I$, with $\card\; J \leq n+1$ the intersection $\mathcal{F}_J$
is non-empty and the graph of a monotone map.
Suppose that $\dim \;\mathcal{F}_I = p$, where 
$0 \leq p \leq n$.
Then, there exists a subset $J \subset I$ with $\card \; J \leq n-p$ such
that $\dim \;\mathcal{F}_J = p$. 
(Note that if $J= \emptyset$ then ${\mathcal F}_J=\Real^n$ by convention, and  
$\dim {\mathcal F}_J=n$.)
\end{lemma}

\begin{proof}
The proof is by induction on $p$.
If $p=n$, then the lemma trivially holds.

Suppose that the claim holds for all dimensions strictly larger than $p$.
Then, there exist a subset $J' \subset I$ (possibly empty), with
$\dim \; \mathcal{F}_{J'} > p$ (noting that if $J' = \emptyset$, then $\dim \; \mathcal{F}_{J'} = n$),
and  $i \in I$ such that $\dim \; \mathcal{F}_{J' \cup\{i\}} = p$.
By the induction hypothesis there exists a subset $J'' \subset J'$, with $\card\; J'' < n-p$,
such that
$\dim \;
\mathcal{F}_{J''}
= \dim \;\mathcal{F}_{J'}
$.

Since
$\displaystyle{
\mathcal{F}_{J'} \subset
\mathcal{F}_{J''}
}
$,
there exists an open definable subset $U \subset \Real^n$ such that
\begin{enumerate}
\item
$\displaystyle{
U \cap \mathcal{F}_{J'} = U \cap \mathcal{F}_{J''}
}
$ and
\item
$\dim \; (U \cap \mathcal{F}_{J'} \cap \mathbf{F}_i) = p$.
\end{enumerate}
Then 
taking $J = J'' \cup \{i\}$, 
$\card \; J \leq n-p$, and, since $n-p \leq n+1$,
the intersection $\mathcal{F}_J$ is the graph of a monotone map by the conditions of the 
lemma.
This proves the lemma because $\mathcal{F}_J$, being a regular cell,
has the same local dimension at each point.
\end{proof}

\begin{proof}[Proof of Theorem \ref{the:monotone_helly}]
The proof is by a double induction on $n$ and $s$.
For $n=1$ the theorem is true for all $s$, since it is just Helly's
theorem in dimension $1$.

Now assume that the statement is true in dimension $n-1$ for all $s$.
In dimension $n$, for $s \leq n+1$, there is nothing to prove.
Assume that the theorem is true in dimension $n$ for at least $s-1$ sets.

The proof is in several steps.
\begin{enumerate}
\item
We first prove that
$\displaystyle{
\mathcal{F}_I
}
$
is non-empty. The proof of this fact is adapted from the classical proof of the topological
version of Helly's theorem (also due to Helly \cite{Helly}).

According to Lemma \ref{lem:intermediate} there exists a family
$\mathcal{F}' = (\mathbf{F}_i')_{i \in I}$
consisting of 
compact
definable 
sets, such that for each $J \subset I$ we have
$$
\displaylines{
\HH_*(\mathcal{F}_J',\mathbb{Z})
\cong
\HH_*(\mathcal{F}_J,\mathbb{Z}).
}
$$
Thus, it suffices to prove that $\mathcal{F}_I'$ is non-empty.
Suppose that $\mathcal{F}_I'$ is empty.
Then, there exists a smallest sub-family, $(\mathbf{F}_j')_{j \in [p]}$,
for some $p$ with $n+2 \leq p \leq s$, such that $\mathcal{F}_{[p]}'$ is empty,
and for each proper subset $J \subset [p]$ the intersection $\mathcal{F}_{J}'$ is non-empty.

Using the induction hypothesis on $s$, applied to the family $(\mathbf{F}_j)_{j \in J}$, 
for each $J \subset I$ with 
$\card \; J < \card \; I = s$, we conclude that
$\mathcal{F}_J$ is the graph of a monotone map and hence acyclic.
But then the set $\mathcal{F}'_J$ is also acyclic since it has the same
singular homology groups as $\mathcal{F}_J$.
Consider the nerve simplicial complex of the family $(\mathbf{F}_j')_{j \in [p]}$.
It has the homology of the $(p-2)$-dimensional sphere $\mathbf{S}^{p-2}$
being isomorphic to the simplicial complex of the boundary of a $(p-1)$-dimensional simplex.
Therefore, the union
$\displaystyle{
\bigcup_{i \in [p]} \mathbf{F}_i'
}
$
also has the homology of $\mathbf{S}^{p-2}$, which is impossible since $p-2 \geq n$.
Thus, $\mathcal{F}_I'$ is non-empty, and hence $\mathcal{F}_I$ is non-empty as well.

\item
We next prove that $\mathcal{F}_I$ is connected.
If not, let $\mathcal{F}_I= B_1 \cup B_2$, where the sets $B_1,B_2$ are non-empty, disjoint
and closed in $\mathcal{F}_I$.

For any $c \in \Real$ the intersection $\mathcal{F}_I \cap \{x_1 =c\}$, where $x_1$
is a coordinate in $\Real^n$, is either empty or connected,
by Theorem~\ref{th:def_monotone_map} and the induction hypothesis for dimension $n-1$.
Hence, $B_1$ and $B_2$ must lie on the opposite sides of a hyperplane
$\{ x_1=c \}$ 
for some $c\in \Real$, with
$$B_1 \cap \{x_1=c\} = B_2 \cap \{x_1=c\} = \emptyset.$$
Now, for every $J \subset I$, such that $\card \; J \leq n$, the intersection
$\mathcal{F}_J$ is the graph of a monotone
map by the conditions of the theorem, and contains both $B_1$ and $B_2$.
Hence $\mathcal{F}_J$ meets the hyperplane $\{ x_1 = c \}$, and,
by Theorem~\ref{th:def_monotone_map}, the intersection
$\mathcal{F}_J \cap \{x_1 =c\}$ is a graph of a monotone map.
Applying the induction hypothesis in dimension $n-1$, to the family
$(\mathbf{F}_i \cap \{x_1 =c\})_{i \in I}$
we obtain that $\mathcal{F}_I \cap \{x_1=c\}$ is non-empty,  which is a contradiction.


\item
We next prove that 
$\mathcal{F}_I$ is the graph of a quasi-affine map.


Let $\dim\; \mathcal{F}_I=p$.
If $p 
=n$,
then $\mathcal{F}_I$ is an non-empty, open, bounded, definable set and
is automatically the graph of a quasi-affine map (cf. Remark \ref{rem:quasi-affine}). 
So we can assume that $p < n$.

By Lemma \ref{lem:intermediate2},  there exists $J \subset I$ with $\card \; J \leq n-p$ such
that $\dim \;\mathcal{F}_J = p$.
By the assumption of the theorem, $\mathcal{F}_J$ is the graph of a monotone map, in particular,
that map is quasi-affine.
Since $p <n$, there exists $i \in J$ such that $m:= \dim\; \mathbf{F}_i < n$.
Assume $\mathbf{F}_i$ to be the graph of a monotone map defined on the semi-monotone subset of
the coordinate subspace $T$.
Then $\dim \; T = m < n$.

Let $\rho_T:\> \Real^n \to T$ be the projection map.
Consider the family
\[
\mathcal{F}'' := (\rho_T(\mathbf{F}_j \cap \mathbf{F}_i))_{j \in I}
.
\]

Every intersection of at most $m+1$ members of $\mathcal{F}''$ is the image under $\rho_T$ of the
intersection of at most $m+2 \leq n+1$ members of $\mathcal{F}$.
By the assumption of the theorem, each intersection of at most $m+2 \leq n+1$ members of $\mathcal{F}$
is the non-empty graph of a monotone map.
Then, by Theorem~\ref{th:proj}, every intersection of at most $m+1$ elements of
$\mathcal{F}''$ is non-empty and is either the graph of a monotone map or a semi-monotone set.
The case when all intersections are semi-monotone sets is trivial, so assume that some of them
are graphs of a monotone maps.
Applying the induction hypothesis (with respect to $n$) to the family $\mathcal{F}''$
we obtain that the intersection,
$
\mathcal{F}''_I
$
is a graph of a monotone map defined on some semi-monotone subset
$ U \subset L$ where $L$ is a coordinate subspace of $T$, and hence
$
\mathcal{F}_I
$
is the graph of a definable
map defined on
$U$.
This, together with the fact that
$\mathcal{F}_I$ is contained in the graph $\mathcal{F}_J$ of a quasi-affine map,
having the same dimension, implies that $\mathcal{F}_I$ is also the graph of a quasi-affine map.

\item
We now prove that 
$\mathcal{F}_I $ is the graph of a monotone map. We have already shown that 
$\mathcal{F}_I $ is non-empty and connected.
Also, the intersection of $\mathcal{F}_I$ with every affine coordinate subspace is connected
by applying again the induction hypothesis.
It now follows from
Definition~\ref{def:def_monotone_map} that
$
\mathcal{F}_I
$ is the graph of a monotone map.

\item
Finally, we prove  that if $\dim \; \mathcal{F}_J \geq d$
for each $J \subset I$, with $\card\; J \leq n+1$, then $\dim \; \mathcal{F}_I \geq d$.
If $d=n$, then $\mathcal{F}_I$ is non-empty and open, and
hence $\dim \; \mathcal{F}_I = d$. So we can assume that $d < n$.

Since $d < n$, there exists $i \in I$, such that $m:= \dim \mathbf{F}_i  < n$.
Let $T \subset \Real^n$ be a coordinate
subspace such that $\mathbf{F}_i$ is a graph of a map over a non-empty semi-monotone
subset of $T$, and let $\dim \;T = m$.
Consider the family $(\rho_T(\mathbf{F}_j \cap \mathbf{F}_i))_{j \in I}$.
By assumption of the theorem and 
Theorem~\ref{th:proj}
we have that for every subset $J \subset I$, with $\card \; J \leq n$,
the family $(\rho_T(\mathbf{F}_j \cap \mathbf{F}_i))_{j \in I}$ consists
of graphs of monotone maps, and every finite intersection of at most
$m+1 \leq n$ 
of these sets is non-empty and also the graph of monotone map having
dimension at least $d$.
Using the induction hypothesis with respect to $n$, we conclude that
$$\dim \; \bigcap_{j \in I} \rho_T(\mathbf{F}_j \cap \mathbf{F}_i) \geq d.$$
It follows that $\dim\; \mathcal{F}_I \geq d$.
\end{enumerate}
\end{proof}

\begin{remark}
\label{rem:final_remark}
It is possible to generalize Theorem \ref{the:monotone_helly} 
slightly by requiring only that all
members of the family $\mathcal{F}$ be contained in the graph of some fixed 
monotone map of dimension $n$ (rather than in $\Real^n$ as in Theorem
\ref{the:monotone_helly}). However, since this would unduly 
complicate the statement of the theorem  we preferred not to make this
slight extension.
\end{remark}

\bibliographystyle{plain}
\bibliography{master}

\def\cprime{$'$}
\begin{thebibliography}{1}

\bibitem{BGV2012}
Saugata Basu, Andrei Gabrielov, and Nicolai Vorobjov.
\newblock Monotone functions and maps.
\newblock {\em Revista de la Real Academia de Ciencias Exactas, Fisicas y
  Naturales. Serie A. Matematicas}, 107(1):5--33, 2013.

\bibitem{BGV}
Saugata Basu, Andrei Gabrielov, and Nicolai Vorobjov.
\newblock Semi-monotone sets.
\newblock {\em J. Eur. Math. Soc. (JEMS)}, 15(2):635--657, 2013.

\bibitem{Michel2}
Michel Coste.
\newblock {\em An introduction to o-minimal geometry}.
\newblock Istituti Editoriali e Poligrafici Internazionali, Pisa, 2000.
\newblock Dip. Mat. Univ. Pisa, Dottorato di Ricerca in Matematica.

\bibitem{Grunbaum}
Branko Gr{\"u}nbaum.
\newblock The dimension of intersections of convex sets.
\newblock {\em Pacific J. Math.}, 12:197--202, 1962.

\bibitem{Helly}
Eduard Helly.
\newblock \"{U}ber {S}ysteme von abgeschlossenen {M}engen mit
  gemeinschaftlichen {P}unkten.
\newblock {\em Monatsh. Math. Phys.}, 37(1):281--302, 1930.

\bibitem{Katchalski71}
Meir Katchalski.
\newblock The dimension of intersections of convex sets.
\newblock {\em Israel J. Math.}, 10:465--470, 1971.

\bibitem{Katchalski78}
Meir Katchalski.
\newblock Reconstructing dimensions of intersections of convex sets.
\newblock {\em Aequationes Math.}, 17(2-3):249--254, 1978.

\bibitem{Radon}
Johann Radon.
\newblock Mengen konvexer {K}\"orper, die einen gemeinsamen {P}unkt enthalten.
\newblock {\em Math. Ann.}, 83(1-2):113--115, 1921.

\bibitem{Spanier}
Edwin~H. Spanier.
\newblock {\em Algebraic topology}.
\newblock McGraw-Hill Book Co., New York, 1966.

\end{thebibliography}
\end{document}